\documentclass[11pt, reqno]{amsart}
\usepackage{fullpage}
\usepackage[titletoc]{appendix}

\newif\ifHideFoot
\HideFoottrue  % This line can be left alone.
\HideFootfalse  % Comment this line to hide footnotes

\usepackage{amssymb,amsmath,amsthm,amscd,mathrsfs,graphicx, color}
\usepackage[cmtip,all,matrix,arrow,tips,curve]{xy}
\usepackage{hyperref}
\usepackage[usenames,dvipsnames]{xcolor}
\usepackage{xypic}
\hypersetup{colorlinks=true,citecolor=ForestGreen,linkcolor=Maroon,urlcolor=NavyBlue}
\usepackage{mathtools}
\usepackage{mathpazo}

\numberwithin{equation}{section}
\newtheorem{teo}{Theorem}[section]
\newtheorem{pro}[teo]{Proposition}
\newtheorem{lem}[teo]{Lemma}
\newtheorem{cor}[teo]{Corollary}

\newtheorem{teoalpha}{Theorem}

\newtheorem{coralpha}{Corollary}
\newtheorem{conalpha}{Conjecture}

\theoremstyle{definition}

\newtheorem{exa}[teo]{Example}

\theoremstyle{remark}
\newtheorem{rem}[teo]{Remark}

\ifHideFoot

\newcommand{\Yano}[1]{}
\newcommand{\Jeff}[1]{}
\newcommand{\Charles}[1]{}

\else 

\newcommand{\marg}[1]{\normalsize{{
\color{red}\footnote{{\color{blue}#1}}}{\marginpar[\vskip
-.25cm{\color{red}\hfill\tiny\thefootnote$\implies$}]{\vskip
-.2cm{\color{red}$\impliedby$\tiny\thefootnote}}}}}
\newcommand{\Yano}[1]{\marg{(Yano) #1}}
\newcommand{\Jeff}[1]{\marg{(Jeff) #1}}
\newcommand{\Charles}[1]{\marg{(Charles) #1}}

\fi

\DeclareMathOperator{\coniveau}{N}

\def\mmu{{\pmb\mu}}

\def\inv{^{-1}}

\def\cx{{\mathbb C}}

\def\rat{{\mathbb Q}}

\renewcommand{\bar}[1]{{\overline{#1}}}

\DeclareMathOperator{\aut}{Aut}

\title[Algebraic normal functions]{Normal functions for algebraically trivial cycles are algebraic for arithmetic reasons}

\author{Jeffrey D. Achter}
\address{Colorado State University, Department of Mathematics,
Fort Collins, CO 80523,
USA}
\email{j.achter@colostate.edu}

\author{Sebastian Casalaina-Martin }
\address{University of Colorado, Department of Mathematics, 
Boulder, CO 80309, USA }
\email{casa@math.colorado.edu}

\author{Charles Vial}
\address{Universit\"at Bielefeld, Fakult\"at f\"ur Mathematik, Germany} 
\email{vial@math.uni-bielefeld.de}

\thanks{
The second author was partially
supported by an NSA grant (H98230-16-1-0053) and Simons Foundation grant
(581058). }

\begin{document}

\vspace{-5pt}
\begin{abstract}
For families of smooth complex projective varieties we show that
normal functions arising from algebraically trivial cycle classes are
algebraic, and defined over the field of definition of the family.  
In particular, the zero loci
of those functions are algebraic and defined over such a field of definition.
This proves a conjecture of Charles.
\end{abstract}

\maketitle

\section{Introduction}

Let $\mathsf f:\mathsf X\to \mathsf B$ be a smooth
 surjective projective morphism of complex algebraic
manifolds, let $n$ 
be an integer, and let $\mathsf J^{2n+1}(\mathsf X/\mathsf B)\to \mathsf B$ be the
$(2n+1)$-st relative Griffiths intermediate Jacobian.
If $\mathsf Z\in
\operatorname{CH}^{n+1}(\mathsf X)$ is an algebraic cycle class such that for
every $\mathsf b\in \mathsf B$ the Gysin fiber $\mathsf Z_{\mathsf b}$  is
algebraically (resp.~homologically) trivial, then there is an
associated holomorphic function 
$$
\nu_{\mathsf Z}:\mathsf B\longrightarrow \mathsf J^{2n+1}(\mathsf X/\mathsf B),
\qquad
\nu_{\mathsf Z}(\mathsf b)=\mathsf {AJ}_{\mathsf X_{\mathsf b}}(\mathsf
Z_{\mathsf b}),
$$
where  $\mathsf {AJ}_{\mathsf X_{\mathsf b}}:\operatorname{CH}^{n+1}(\mathsf
X_{\mathsf b})_{hom}\to \mathsf J^{2n+1}(\mathsf X_{\mathsf b})$ is the
Abel--Jacobi map on homologically trivial cycles in the fiber $\mathsf
X_{\mathsf b}$.    Such a function is  called an algebraically
  motivated (resp.~motivated) normal function motivated by the
cycle class $\mathsf Z$.  

More generally, let $\mathsf B$ be a complex manifold, and let $\mathcal H$ be a
variation of pure negative weight integral Hodge structures
over $\mathsf B$.  In \cite{saitoJAG96}, Saito defines the notion of an
admissible normal function as a holomorphic
section
$\nu:\mathsf B\to \mathsf J(\mathcal H)$ of the associated family of
generalized intermediate Jacobians $\mathsf J(\mathcal H)\to \mathsf
B$ that satisfies a version of Griffiths horizontality and has
controlled asymptotic behavior near the boundary 
(see \emph{e.g.}, \cite{BP-13-Comp}).
Despite the transcendental nature of the definition of admissible normal
functions, 
there is the following conjecture due to Green and Griffiths (\emph{e.g.}, \cite[p.883]{BP-09-Annals},  \cite[Conj.~1]{charlesIMRN10},  \cite[Conj.~1.1]{schnell-Inv-12}, \cite[p.1914]{BP-13-Comp}):

\begin{conalpha}[{Green--Griffiths}]\label{CJ:GG}
The zero locus of an admissible normal function  on a complex algebraic manifold
is algebraic.
\end{conalpha}

Proofs of this conjecture were given in a series of papers\,:  $\dim  \mathsf B=1$
\cite[Cor.~1]{saitoArXiv08}, \cite[Thm.~4.5]{BP-09-Annals}, $\dim \mathsf B\ge
1$ 
\cite[Thm.~C]{schnell-Inv-12},  \cite[Cor.~1.3]{BP-13-Comp} (see also
S\'em.~Bourbaki \cite{charlesBourbaki12}). 
In this paper, we are interested in algebraic and arithmetic questions concerning motivated normal functions. 
First, 
for algebraically motivated (resp.~motivated) normal
functions, if $\mathsf X$, $\mathsf B$, $\mathsf f$, and $\mathsf Z$ are all
defined over a subfield $F\subseteq \mathbb C$, we say that the normal function 
$\nu_{\mathsf Z}$ is algebraically $F$-motivated (resp.~$F$-motivated),  
and it is natural to ask whether the zero locus of $\nu_{\mathsf Z}$ in $\mathsf
B$ is also defined over~$F$ \cite[p.2284]{charlesIMRN10}\, (see also \cite[Conj.~81]{KP-Exp-11}):

\begin{conalpha}[{Charles}]\label{CJ:ANFMK} Let  $F\subseteq
\mathbb C$ be a subfield. 
The zero locus of an algebraically $F$-motivated (resp.~$F$-motivated)
normal function  is  algebraic and defined over $F$.
\end{conalpha}

Several partial results are known.  
Regarding the $F$-motivated case of Conjecture~\ref{CJ:ANFMK}, 
a special case of a result of Saito \cite[Cor.~1]{Saito-normal} on admissible normal functions implies that if an irreducible component of 
the zero locus of an $F$-motivated normal function contains a point of $\mathsf B$ that is defined over $F$, then  the entire component of the zero locus is defined over $F$ (see also  \cite[Thm.~3]{charlesIMRN10},  \cite[Thm. 89]{KP-Exp-11}).
Regarding the algebraically $F$-motivated case of Conjecture~\ref{CJ:ANFMK},
Kerr and Pearlstein have shown in \cite[Con.~81, $\widetilde{\mathfrak
{ZL}}(D,1)_{alg}$, Thm.~88]{KP-Exp-11} that the zero locus of an algebraically
$F$-motivated normal function is an algebraic
subset of $\mathsf B$ defined over a finite extension of~$F$.  
All of the aforementioned results take as a starting point the validity of Conjecture~\ref{CJ:GG}.
\medskip

In this paper  we directly prove Conjecture
\ref{CJ:ANFMK} in the algebraically $F$-motivated
case.  (In particular, we do not rely on earlier work on Conjecture
\ref{CJ:GG}.)
In fact, we 
prove a stronger result, namely that 
algebraically $F$-motivated normal
functions are themselves algebraic and defined
over $F$\,:

\setcounter{teoalpha}{0}
\begin{teoalpha}\label{T:ANF-MK} Let 
$\mathsf f:\mathsf X\to \mathsf B$ be a smooth 
surjective
projective
morphism of complex algebraic manifolds (not necessarily connected), 
let $n$ be a nonnegative
integer, let $\mathsf J^{2n+1}(\mathsf X/\mathsf B)\to \mathsf B$ be
the $(2n+1)$-st relative Griffiths intermediate Jacobian.  There is a relative
algebraic complex subtorus $\mathsf J^{2n+1}_a(\mathsf X/\mathsf B)\subseteq
\mathsf J^{2n+1}(\mathsf X/\mathsf B)$ over $\mathsf B$ such that for very
general $\mathsf u\in \mathsf B$ the fiber $\mathsf J_a^{2n+1}(\mathsf X/\mathsf
B)_{\mathsf u}\subseteq \mathsf J^{2n+1}(\mathsf X_{\mathsf u})$ is the image
$\mathsf J^{2n+1}_a(\mathsf X_{\mathsf u})$ of the Abel--Jacobi map 
$\mathsf {AJ}_{\mathsf X_{\mathsf u}}:\operatorname{A}^{n+1}(\mathsf X_{\mathsf
u})\to \mathsf J^{2n+1}(\mathsf X_{\mathsf u})$, 
and for any algebraic cycle class 
$\mathsf Z\in \operatorname{CH}^{n+1}(\mathsf X)$
such that for
every $\mathsf b\in \mathsf B$ the Gysin fiber $\mathsf Z_{\mathsf b}$  is
algebraically trivial\,:

\begin{enumerate}
\item The normal function  $
\nu_{\mathsf Z}:\mathsf B\to \mathsf J^{2n+1}(\mathsf X/\mathsf B)$ has image
contained in $\mathsf J^{2n+1}_a(\mathsf X/\mathsf B)$ and is an algebraic map. 

\item If, moreover, $\mathsf X$, $\mathsf B$, $\mathsf f$, and $\mathsf Z$ are
all defined over a field $F\subseteq \mathbb C$, then so are $\mathsf
J^{2n+1}_a(\mathsf X/\mathsf B)$ and the morphisms $\mathsf J^{2n+1}_a(\mathsf
X/\mathsf B)\to \mathsf B$ and $\nu_{\mathsf Z}$.

\end{enumerate}
\end{teoalpha}

See Remark \ref{R:Jnotation} for a caution about the
notation $\mathsf J^{2n+1}_a(\mathsf X/\mathsf B)$, and see the notation and conventions below for a reminder on very general points. 
Conjecture \ref{CJ:GG} in the  algebraically motivated case  follows 
immediately from Theorem \ref{T:ANF-MK}(1), and in this way we obtain a short
proof of this case of the conjecture.  
Conjecture \ref{CJ:ANFMK} in the algebraically $F$-motivated case  follows
immediately from Theorem \ref{T:ANF-MK}(2).   In summary, we have\,:
\begin{coralpha}\label{C:ZL}
Let  $F\subseteq
\mathbb C$ be a subfield. 
The zero locus of an algebraically $F$-motivated 
normal function  is  algebraic and defined over $F$.
\end{coralpha}

We review some special easy cases of Theorem \ref{T:ANF-MK}, with an eye toward explaining why their generalization is not  immediate.
In the case of $n+1=1$, \emph{i.e.},  of $\mathsf{Pic}^0_{\mathsf X/\mathsf B}$, and
$n+1=\dim_{\mathsf B}\mathsf X$, \emph{i.e.},  of $\mathsf{Alb}_{\mathsf X/\mathsf B}$,
it is well known that algebraically $F$-motivated  normal functions
are algebraic and defined over $F$ (\emph{e.g.}, \cite[Thm.~VI.3.3]{FGA}, \cite[Def.~4.6
and Thm.~4.8]{kleimanPIC}).  
In the case where $\mathsf B$ is quasiprojective, $\mathsf X=\mathsf B\times
\mathsf Y$ for some smooth projective complex manifold $\mathsf Y$, and $\mathsf
f:\mathsf X\to \mathsf B$ is the first projection, part (1) of the theorem is
elementary by embedding $\mathsf B$ in a smooth complex projective manifold
$\bar {\mathsf B}$, extending the cycle class $\mathsf Z$ to a cycle class $\bar
{\mathsf Z}$ on $\bar{\mathsf B}\times \mathsf Y$, and obtaining a normal
function $\nu_{\bar {\mathsf Z}}:\bar{\mathsf B}\to \bar {\mathsf B}\times
\mathsf J^{2n+1}_a(\mathsf Y)$ that is a holomorphic map between complex
projective manifolds.  
From our work in  \cite{ACMVdmij}, one can then easily deduce (2) of the theorem
in this case, as well.   The difficulty in using the same strategy to prove part
(1) of the theorem in general is twofold. First,  the family $\mathsf f:\mathsf
X\to \mathsf B$ may not extend to a smooth family over $\bar {\mathsf B}$, in
which case it is difficult to know how to extend $\mathsf J^{2n+1}(\mathsf
X/\mathsf B)$ and $\nu_{\mathsf Z}$ to the boundary.  Second, even if one can
extend  $\mathsf f:\mathsf X\to \mathsf B$  to a  smooth family $\bar{\mathsf
f}:\bar {\mathsf X}\to\bar {\mathsf B}$, the geometric coniveau of the family
can jump along a countable union of algebraic subsets  of $\bar{\mathsf B}$,
 and
so there is no obvious algebraic target $\mathsf J^{2n+1}_a(\bar {\mathsf
X}/\bar{\mathsf B})$ for an extended normal function.  

One faces similar difficulties in trying to prove  Conjecture \ref{CJ:GG},  and
the approach taken  in   \cite{saitoArXiv08,BP-09-Annals, schnell-Inv-12}
overcomes these complications by constructing N\'eron models for the relative
intermediate Jacobians (see also \cite{GGK-comp-10}) that provide manageable
targets for  extending admissible normal functions.  
In the special case where $\dim \mathsf B=1$, Schnell and Kerr independently
communicated to us arguments using these techniques   to prove part (1) of 
the
theorem, up to replacing the normal function $\nu_{\mathsf Z}$ with $M\cdot
\nu_{\mathsf Z}$ for some integer $M$,  depending on $\mathsf Z$.  It appears
however that it would be difficult to extend these arguments to the case where
$\dim \mathsf B\ge 2$.  It also appears it would be difficult to use these
techniques to prove part (2) of the theorem regarding the field of definition, even in the case where $\dim
\mathsf B=1$.
\medskip

The starting point of our proof consists in showing that, for a smooth projective variety $X$ defined over a subfield $K\subseteq \cx$, the kernel of the Abel--Jacobi map restricted to algebraically trivial cycles defined over $K$ is independent of the choice of field embedding $K\subseteq \cx$. 
This is embodied in Corollary~\ref{Cor:fieldemb}\,; in fact, a
stronger result is proved in Proposition~\ref{P:DNF-K} where it is
shown that the distinguished model of \cite{ACMVdmij} does not depend
on a choice of field embedding.  The proof uses in an essential way
the fact proven in \cite{ACMVabtriv} that algebraically trivial cycles
defined over $K$ are parameterized by abelian varieties, and builds on
our previous work \cite{ACMVdmij}. Consequently, the relevant material of \cite{ACMVdmij} 
is reviewed in \S \ref{S:DMIJ-Pre}. An important consequence of Proposition~\ref{P:DNF-K} 
is that an algebraically $F$-motivated normal function vanishes at a very general point 
if and only if it vanishes on a Zariski open subset, and is therefore identically zero\,; see Example~\ref{E:NF-CCFam} and
 Remarks~\ref{R:Geom-Gen-Fib} and~\ref{R:FC-NF}.

In fact, the initial step of our strategy is to consider a very general fiber $\mathsf X_{\mathsf u}$ and, 
thanks to Proposition \ref{P:DNF-K}, to
descend the image of the Abel--Jacobi map
$\mathsf J^{2n+1}_a(\mathsf X_{\mathsf u})$ for this fiber to an abelian variety
over the generic point of   $\mathsf B$, which admits a natural section related to
the normal function. 
 We then spread this abelian variety and section to a
Zariski open subset of $\mathsf B$, all defined over $F$. In Theorem \ref{T:ANF-MK-LG}, we compare this abelian scheme together with the induced section to the analytic normal function.
 This is achieved through comparing the related variations of Hodge structures via an algebraic correspondence defined over $K$, provided by Theorem~\ref{T:DMIJ-0}.  
There is a technical point here, that the correspondence only identifies the integral Hodge structures, as well as our algebraic section and the analytic normal function,   up to an integer multiple $M$\,; in Theorem \ref{T:ANF-MK-LG} we show that the image of the   morphism of abelian varieties induced by the correspondence,  inside the Griffiths intermediate Jacobian, is in fact the spread of our distinguished model, and that the algebraic section and analytic normal function are identified.
 The final step is to
extend this to all of $\mathsf B$ (\S \ref{S:Ext-DM-DNF}).  
We extend the relative algebraic torus  over the generic points of 
codimension-$1$ boundary loci by using the good reduction of $\mathsf X$ and 
the N\'eron--Ogg--Shafarevich
criterion, 
and then extend over codimension-$2$ loci using the Faltings--Chai Extension
Theorem.
The normal function is handled separately at each step.  In short, rather than
having to worry about extending admissible normal functions to projective
compactifications in order to obtain algebraicity as a consequence of Chow's
theorem, we extend algebraic maps defined over $F$ on a Zariski open subset of
$\mathsf B$ to all of $\mathsf B$, and in this way  also manage to maintain
control over the field of definition.    

\medskip

In forthcoming work \cite{ACMVfunctorial} we will study the notion of regular  homomorphisms in the relative setting\,; Theorem~\ref{T:ANF-MK}  shows that the Abel--Jacobi map provides such a relative regular homomorphism.

Finally,
although our results are algebraic in nature and only concern algebraically trivial cycle classes, 
there are important instances of families of varieties for which homological and algebraic equivalence
of cycles in certain codimensions agree and for which the corresponding intermediate 
Jacobians are algebraic \cite{BlSr83}. For example, a direct application of Theorem \ref{T:ANF-MK} 
concerns codimension-2 cycles on uniruled threefolds\,:

\begin{coralpha}
Let $\mathsf f: \mathsf X\to \mathsf B$ be a smooth projective family of uniruled threefolds, 
defined over a field $F\subseteq \cx$, and let $\mathsf Z\in \operatorname{CH}^2(\mathsf X)$ 
be a cycle class defined over $F$ that is fiber-wise homologically trivial. Then the analytic 
 normal function $\nu_{{\mathsf Z}}$ is algebraic, and defined over $F$. 
 In particular, its zero-locus is an algebraic sub-variety of $\mathsf B$ defined over $F$.
\end{coralpha}

In the opposite direction, as a specific example of a case where it is
not clear how to apply our results to  the $F$-motivated case of
Conjecture \ref{CJ:ANFMK}, one can consider families of Calabi--Yau
threefolds\,; specifically for very general quintic threefolds, Voisin
\cite[Thm.~2,3]{VoisinGriffiths} has shown that the image of the
Abel--Jacobi map on \emph{algebraically} trivial codimension-$2$ cycle
classes is trivial, while the image of the Abel--Jacobi map on
\emph{homologically} trivial codimension-$2$ cycle classes is a
countable abelian group of infinite rank.  In general, it seems to
us that it would be interesting to construct a canonical arithmetic
structure on the image of all homologically trivial cycles under the
Abel--Jacobi map.  Such a construction is a necessary first step for
extending our methods to normal functions motivated by homologically trivial cycles.

\subsection*{Acknowledgements} 
The second author would like to thank Alena Pirutka for a conversation 
regarding the algebraicity of motivated normal functions,  which led to this
investigation.  He would also like to thank Christian Schnell and Matt Kerr for
explaining the connection to their previous work, which was very helpful in
formulating our results.    We thank the referee for helpful suggestions.

\subsection*{Notation and conventions}
Let $X$ be a scheme of finite type over a field $F\subseteq \mathbb C$.
We denote by $X(F)$ the set of $F$-morphisms $\operatorname{Spec}F\to X$.  
We denote by $\mathsf X=X_{an}$ the associated complex analytic space.  After
identifying the sets  
\begin{align*}
\mathsf X&=\{\mathfrak p\in |X_{\mathbb C}|: \mathfrak p \text{ is closed in the
underlying topological space }  |X_{\mathbb C}|\}\\
&=\{x\in X_{\mathbb C}(\mathbb C): x(\operatorname{Spec}\mathbb C) \text{ is
closed in $|X_{\mathbb C}|$}\},
\end{align*} 
we say that $\mathsf x\in \mathsf X$ is \emph{$F$-very general} if the
corresponding morphism $x:\operatorname{Spec}\mathbb C\to X$ has image
a generic point of $|X|$.  If $F$ is countable, then so is the collection of
all closed algebraic subsets of~$\mathsf X$ that are defined
over $F$ and not equal to $\mathsf X$\,;  any $F$-very general point $\mathsf x$
is in the
complement of the union of these closed algebraic subsets.  If $\mathsf Y$ is merely a complex algebraic variety, then a very general point of $\mathsf Y$ is an $F$-very general point for some field of definition $F$ of $\mathsf Y$ which is of finite type over $\rat$.

A \emph{variety} over a field is a geometrically reduced separated scheme of
finite type over that field.  
Given a smooth projective variety  $X$ over a field $F\subseteq \mathbb C$ we
denote by $\operatorname{CH}^i(X)$ (resp.~$\operatorname{CH}^i(\mathsf X)$) the
Chow group of codimension-$i$ algebraic cycle classes on $X$ (resp.~$\mathsf
X$), and by $\operatorname{A}^i(X)$ (resp.~$\operatorname{A}^i(\mathsf X)$) 
the subgroup of algebraically trivial algebraic cycle classes on $X$ (resp.~$\mathsf
X$).

For the remainder of the paper, the domain of the
Abel--Jacobi map $\mathsf {AJ}:\operatorname{A}^{n+1}(\mathsf X)\to
\mathsf J^{2n+1}(\mathsf X)$ is the group of algebraically trivial algebraic
cycle classes.  We denote by $\mathsf J^{2n+1}_a(\mathsf X)$ the image
of this Abel--Jacobi map, and by $\mathsf i^{2n+1}_{a,\mathsf X}:\mathsf J^{2n+1}_a(\mathsf X)\to \mathsf
J^{2n+1}(\mathsf X)$ the natural inclusion.

\section{Distinguished models of intermediate Jacobians \\
	and distinguished normal functions} \label{S:DMIJ-Pre}

In this section we recall some results from \cite{ACMV14, ACMVdmij} regarding
descending intermediate Jacobians to a field of definition.

\subsection{Distinguished models of intermediate Jacobians}\label{S:DMIJ-rev}

We start by recalling the main result of \cite{ACMVdmij}. We note that while in \S \ref{S:DMIJ-rev}, \ref{S:DNF} we  work
over a field $K\subseteq \mathbb C$, starting from \S \ref{S:DM-DNF} we will
implement these results in the case where the field $K$ is the residue field of
the generic point of the integral base $B$ of a smooth projective family $f:X\to
B$, all defined over a field $F\subseteq \mathbb C$.

\begin{teo}[{Distinguished models \cite[Thm.~1]{ACMVdmij}}]\label{T:DMIJ-0}
Suppose $X$ is a smooth projective variety over a field $K\subseteq \cx$, with
associated complex analytic space $\mathsf X$, and
let $n$ be a nonnegative integer. Then $\mathsf J^{2n+1}_a(\mathsf X)$,  the
algebraic complex torus that is the image of the Abel--Jacobi map $\mathsf {AJ}
:
\operatorname{A}^{n+1}(\mathsf X) \to \mathsf J^{2n+1}(\mathsf X)$,  
admits a distinguished model $J^{2n+1}_{a,X/K}$ over $K$ such that the
Abel--Jacobi map is 
$\operatorname{Aut}(\cx/K)$-equivariant.
Moreover, there exist a
correspondence $\Gamma \in \operatorname{CH}^{\dim (J^{2n+1}_{a,X/K})+n}(
J^{2n+1}_{a,X/K}\times_K
X)$  and a positive integer $M$ such that 
the induced
morphism $\mathsf \Gamma_*:\mathsf J^{2n+1}_a(\mathsf X)\to \mathsf
J^{2n+1}(\mathsf X)$ is $M \cdot \mathsf i^{2n+1}_{a,\mathsf X}$\,; \emph{i.e.}, $M$ times the natural inclusion.
\end{teo}

\begin{rem}[Uniqueness of the distinguished model]\label{R:Uniq-DM}
By Chow's rigidity theorem,
an abelian variety $A/\mathbb C$ descends to  at most one model  defined over
$\bar K$.  
On the other hand, an abelian variety $A/\bar K$ may descend to  more than one
model   defined over $K$.  Nevertheless,  since $\mathsf
{AJ}:\operatorname{A}^{n+1}(\mathsf X)\to
\mathsf J_a^{2n+1}(\mathsf X)$ is surjective, 
the algebraic complex torus 
$\mathsf J^{2n+1}_a(\mathsf X)$ admits at most one structure of a variety over
$K$ such that $\mathsf {AJ}$ is 
$\operatorname{Aut}(\cx/K)$-equivariant.
More precisely, setting $J^{2n+1}_a(X_{\mathbb C})$ to be the abelian variety associated 
to the algebraic complex torus $\mathsf J^{2n+1}_a(\mathsf X)$, there an  abelian 
variety $J^{2n+1}_{a,X/K}$ which is unique up to unique isomorphism, such that there is an isomorphism 
$(J^{2n+1}_{a,X/K})_{\mathbb C}\to J^{2n+1}_a(X_{\mathbb C})$ such that the induced action 
of  $\operatorname{Aut}(\mathbb C/K)$ on $\mathsf J^{2n+1}_a(\mathsf X)$ 
makes the Abel--Jacobi map $\operatorname{Aut}(\mathbb C/K)$-equivariant.
 This is
the sense in which  $\mathsf J_a^{2n+1}(\mathsf X)$  admits a
\emph{distinguished model}
over~$K$.
\end{rem}

In \cite[Thm.~1]{ACMVdmij} we show that the correspondence $\Gamma$ in 
Theorem \ref{T:DMIJ-0} induces a morphism
 of complex tori  $\mathsf \Gamma_* : \mathsf J^{2n+1}_a(\mathsf X)\to \mathsf J^{2n+1}(\mathsf X)$
 with image $\mathsf J^{2n+1}_a(\mathsf X)$.  In fact, this morphism
 respects $K$-structures\,:

 \begin{lem}
   \label{L:Gamm-M}
In the situation of Theorem \ref{T:DMIJ-0}, the morphism $\mathsf \Gamma_*:\mathsf J^{2n+1}_a(\mathsf X)\to \mathsf J^{2n+1}_a(\mathsf X)$ 
is induced by an isogeny $\psi:J^{2n+1}_{a,X/K}\to J^{2n+1}_{a,X/K}$
over $K$.
 \end{lem}

 \begin{proof}
   It suffices to show that $\mathsf \Gamma_*$ is
   $\aut(\cx/K)$-equivariant on torsion.
This is achieved  by  identifying the map on $N$-torsion with the map $\Gamma_*:H^1(J^{2n+1}_a(X_{\mathbb C}),\mmu_{N})\to H^1(J^{2n+1}_a(X_{\mathbb C}),\mmu_{N})\subseteq  H^{2n+1}(X_{\mathbb C},\mmu_{N}^{\otimes (n+1)})$ (similar to \cite[(2.3)]{ACMVdmij}).  
   Let $M$ be the exponent of $\psi$, and let $\tilde \psi: J^{2n+1}_{a,X/K}\to J^{2n+1}_{a,X/K}$ 
   be such that $\tilde \psi \circ \psi=M$.   With $\Gamma_{\tilde \psi}$ the correspondence 
   associated to the morphism $\tilde \psi$, let $\Gamma'=\Gamma_{\tilde \psi}\circ \Gamma$.  
   It follows that the induced morphism $\mathsf \Gamma'_*:\mathsf J^{2n+1}_a(\mathsf X)\to \mathsf J^{2n+1}(\mathsf X)$ 
   has image $\mathsf J^{2n+1}_a(\mathsf X)$, and is given as $M \cdot
   \mathsf i^{2n+1}_{a,\mathsf X}$\,; \emph{i.e.}, $M$ times the natural inclusion. 
 \end{proof}

\begin{rem}[Extensions $K\subseteq L\subseteq \mathbb C$]\label{R:DM-L/K}
In the notation of the theorem, suppose  we have an intermediate field
extension $K\subseteq L\subseteq \mathbb C$.  Then the base change
$(J^{2n+1}_{a,X/K})_L$ is the distinguished model for $X_L$.  Indeed,
the distinguished model over $L$ is determined uniquely by the fact
that the Abel--Jacobi map for $\mathsf X$ is
$\operatorname{Aut}(\mathbb C/L)$-equivariant\,; but if the Abel--Jacobi
map is $\operatorname{Aut}(\mathbb C/K)$-equivariant with respect to
the $K$-structure on $J^{2n+1}_{a,X/K}$, then it is
$\operatorname{Aut}(\mathbb C/L)$-equivariant for the $L$-structure on
$(J^{2n+1}_{a,X/K})_L$.   
\end{rem}

\subsection{Distinguished normal functions}\label{S:DNF}
In \cite{ACMV14} we established
some results regarding equivariant regular homomorphisms.  In this section we
recall the consequences of that work in the context of normal functions and the
distinguished model.

Let $X$ be smooth projective variety over $K\subseteq \mathbb C$. 
Given $Z\in \operatorname{CH}^{n+1}(X)$ with the base change $Z_{\mathbb C}$
algebraically trivial, and $\sigma\in \operatorname{Aut}(\mathbb C/K)$, 
we showed in  \cite{ACMV14} that the following diagram commutes\,:
$$
\xymatrix@R=1.3em{
(\operatorname{Spec}\mathbb C)_\cx(\mathbb C)\ar[r]^{w_{ Z_{\mathbb C}}}
\ar@{=}[d]&\operatorname{A}^{n+1}(X_{\mathbb C})\ar@{->}[r]^<>(0.5){AJ}
\ar[d]^<>(0.5){\sigma^*}&J^{2n+1}_a(X_{\mathbb C})(\mathbb C)
\ar[d]^<>(0.5){\sigma(\mathbb C)}\\
(\operatorname{Spec}\mathbb C)_\cx(\mathbb C)\ar[r]^{w_{ Z_{\mathbb
C}}}&\operatorname{A}^{n+1}(X_{\mathbb
C})\ar@{->}[r]^<>(0.5){AJ}&J^{2n+1}_a(X_{\mathbb C})(\mathbb C)\\
}
$$
where $(\operatorname{Spec}\mathbb C)_{\mathbb C}(\mathbb C)=\{\operatorname{Id}_{\mathbb
C}\}$, and $w_{Z_{\mathbb C}}(\operatorname{Id}_{\mathbb C})=Z_{\mathbb C}$.
Indeed the right hand side is the precise meaning of the statement in  Theorem
\ref{T:DMIJ-0} that the Abel--Jacobi map  is $\operatorname{Aut}(\mathbb
C/K)$-equivariant, while the left hand side is elementary (see
\cite[Rem.~4.3]{ACMV14} for more on this).
This corresponds to the commutativity
of the diagram of sets\,:
$$
\xymatrix@R=1.2em{
(\operatorname{Spec}\mathbb C)_{an}\ar[r]^{w_{ \mathsf Z}}
\ar@{=}[d]&\operatorname{A}^{n+1}(\mathsf X)\ar@{->>}[r]^<>(0.5){\mathsf {AJ}}
\ar[d]^<>(0.5){\sigma^*}&\mathsf J^{2n+1}_a(\mathsf X)\ar[d]^<>(0.5){\sigma}\\
(\operatorname{Spec}\mathbb C)_{an}\ar[r]^{w_{ \mathsf
Z}}&\operatorname{A}^{n+1}(\mathsf X)\ar@{->>}[r]^<>(0.5){\mathsf {AJ}}&\mathsf
J^{2n+1}_a(\mathsf X)\\
}
$$
where $w_{\mathsf Z}((\operatorname{Spec}\mathbb C)_{an})=\mathsf Z$, the
complex analytic cycle class associated to $Z$, and   $\sigma:\mathsf J^{2n+1}_a(\mathsf X)\to \mathsf J^{2n+1}_a(\mathsf X) $  is the map of sets induced by $\sigma(\mathbb C)$ in the previous diagram.   
Note that $\mathsf {AJ}\circ w_{\mathsf Z}=\nu_{\mathsf Z}$, the normal function
associated to $\mathsf Z$.
As mentioned  in \cite[Rem.~4.3]{ACMV14}, the commutativity of the diagrams
above implies that $AJ\circ w_{Z_{\mathbb C}}$, and hence $\nu _{\mathsf Z}$,
descend to $K$ to give a morphism 
\begin{equation}\label{E:Def-DNF}
\delta_Z:\operatorname{Spec}K\to J^{2n+1}_{a,X/K}.
\end{equation}
We call this the \emph{distinguished normal function} associated to $Z$.

\begin{rem}[Uniqueness of the distinguished normal function] \label{R:Uni-DNF}
 The distinguished normal function is unique in the sense that given the  
 distinguished model $J^{2n+1}_{a,X/K}$ (unique up to unique isomorphism 
 by Remark \ref{R:Uniq-DM}), there is a unique morphism $\delta_Z:\operatorname{Spec}K\to J^{2n+1}_{a,X/K}$ 
 such that $(\delta_Z)_{an}:(\operatorname{Spec}\mathbb C)_{an}\to \mathsf J^{2n+1}_a(\mathsf X)$ 
 is equal to the analytic normal function $\nu_{\mathsf Z}$.   
\end{rem}

\begin{rem}[Extensions $K\subseteq L\subseteq \mathbb C$]\label{R:DNF-L/K}
In the notation of Theorem \ref{T:DMIJ-0}, suppose  we have an intermediate
field extension $K\subseteq L\subseteq \mathbb C$. 
In light of Remark \ref{R:DM-L/K}, we have a fibered product diagram
$$
\xymatrix@R=1.5em{
J^{2n+1}_{a,X_L/L} \ar[d] \ar[r]& J^{2n+1}_{a,X/K} \ar[d]\\
\operatorname{Spec} L \ar[r] \ar@/^1pc/[u]^{\delta_{Z_L}}& \operatorname{Spec} K
\ar@/_1pc/[u]_{\delta_Z}
}
$$
\end{rem}

\subsection{Review of the construction of the distinguished model}
Because it will be relevant later, we recall the construction of the distinguished model $J^{2n+1}_{a,X/K}$ from \cite{ACMVdmij}.  
The starting point is \cite[Prop.~1.1]{ACMVdmij}, which provides a smooth projective
 geometrically integral curve $C/K$ (admitting a $K$-point)    and a correspondence $\gamma \in
  \operatorname{CH}^{n+1}(C\times_K
  X)$  such that 
    the induced morphism of complex tori
 $
\gamma_*: \mathsf J(\mathsf C)\to \mathsf J^{2n+1}(\mathsf X)
$
has image $\mathsf J^{2n+1}_a(\mathsf X)$.  We thus obtain a short exact 
sequence of algebraic compact complex analytic groups $  0\to  \mathsf P \to \mathsf J(\mathsf C) \to  
\mathsf J^{2n+1}_a(\mathsf X) \to 0$,
where   $\mathsf P$  is defined to be the  kernel.

The next step is to show that $\mathsf P$ descends to $K$. 
For this it suffices to show that for every natural number $N$, 
the $N$-torsion $\mathsf P[N]$ is preserved by $\operatorname{Aut}(\mathbb C/K)$ 
(since torsion is dense in any sub-group scheme
of an abelian variety\,; see \emph{e.g.}, \cite[Lem.~2.3]{ACMVdmij}).   
For this one shows that  $\mathsf P[N]$  is equal to the kernel of the morphism 
$\gamma_*: H^1(C_{\mathbb C},\mmu_{N}) \to  H^{2n+1}(X_{\mathbb C},\mmu_{N}^{\otimes (n+1)})$ 
(see  \cite[(2.3)]{ACMVdmij}), which is  equivariant as it is induced by a correspondence defined over $K$.  
Thus $\mathsf P[N]$ is preserved by $\operatorname{Aut}(\mathbb C/K)$, so that $\mathsf P$ 
descends to a group scheme $P/K$, and consequently  $\mathsf J^{2n+1}_a(\mathsf X)$ 
descends to a model $J^{2n+1}_{a,X/K}$ over $K$, as well.   This is the distinguished model.

\begin{rem} \label{R:DM-ind-C-gg} 
We reiterate here that the distinguished model is unique  up to unique isomorphism 
(see Remark \ref{R:Uniq-DM}), so that  $J^{2n+1}_{a,X/K}$ is in fact independent of 
the curve $C$ and the correspondence $\gamma$ used in the construction. 
  In other words,   given any 
smooth projective
 geometrically integral curve $C'/K$ (admitting a $K$-point)   and a correspondence $\gamma' \in
  \operatorname{CH}^{n+1}(C'\times_K
  X)$  such that 
   the
  induced morphism of
complex tori $\gamma'_*:\mathsf J(\mathsf C')\to \mathsf J^{2n+1}(\mathsf X)$  defines 
a short exact sequence of algebraic compact complex analytic groups $  0\to  \mathsf P' \to \mathsf J(\mathsf C') \to  
\mathsf J^{2n+1}_a(\mathsf X) \to 0$,
the descent datum on $C'_\cx$ defines $J^{2n+1}_{a,X/K}$.
\end{rem}

\section{Changing the embedding $K\subseteq \mathbb C$}
In this section we show that if $K$ is a field of finite transcendence degree over $\mathbb Q$, 
then up to  isomorphism, the distinguished model and distinguished normal function do not depend 
on the embedding $K\subseteq \mathbb C$.  An important consequence is that the normal function 
associated to a fiber-wise algebraically trivial cycle defined over a field $F$ of finite transcendence 
degree over $\mathbb Q$ (\emph{i.e.}, an algebraically $F$-motivated normal function in our terminology) 
vanishes at an $F$-very general point if and only if it vanishes at all $F$-very general points\,; 
see Example~\ref{E:NF-CCFam} and Remark~\ref{R:FC-NF}.

\subsection{The distinguished normal function is independent of the field embedding}
The following proposition complements, in particular, Theorem \ref{T:DMIJ-0} by showing that the distinguished model does not depend on the choice of an embedding $K\subseteq \cx$.

\begin{pro}\label{P:DNF-K}
Let $X$ be a smooth projective variety over  a field $K$ 
of finite
transcendence degree over~$\mathbb Q$,  let $n$ be a nonnegative integer, 
and let $Z\in \operatorname{A}^{n+1}(X)$ be an algebraically trivial
cycle class.  
Let $b_1 ,b_2:K\hookrightarrow \mathbb C$ be two inclusions of fields,
let $L_i =b_i(K)$, and denote by 
$\sigma:\mathbb C\to \mathbb C$ an automorphism inducing a commutative diagram
of
field homomorphisms
\begin{equation*}\label{E:fld-Iso}
\xymatrix@R=.4em{
&L_1\ar@{^(->}[r] \ar[dd]^{b_2b_1^{-1}}&\mathbb C \ar@{-->}[dd]^\sigma\\
K\ar[ru]^{b_1}_{\sim} \ar[rd]_{b_2}^{\sim}&&\\
&L_2\ar@{^(->}[r]&\mathbb C \\
}
\end{equation*}
which exists due to the assumption that $K$ is of finite transcendence
degree over $\mathbb Q$.  

For $i = 1,2$ let 
$X_{b_i}$  be the base change of $X$ over $b_i:\operatorname{Spec}\mathbb C\to \operatorname{Spec}K$, 
with associated complex analytic space $\mathsf X_{\mathsf b_i}$, let $J^{2n+1}_{a,X_{L_i/L_i}}$  be 
the distinguished model of $\mathsf J^{2n+1}_a(\mathsf X_{\mathsf b_i})$  over $L_i$, and let $\delta_{Z_{L_i}}:\operatorname{Spec}L_i\to 
J^{2n+1}_{a,X_{L_i/L_i}}$  be the distinguished normal function.
Let $ J^{2n+1}_a( X_{ b_i})$  be the complex abelian variety
associated to $\mathsf J^{2n+1}_a(\mathsf X_{\mathsf b_i})$.

Then $J^{2n+1}_{a,X_{L_2}/L_2}$ is the pullback of
$J^{2n+1}_{a,X_{L_1}/L_1}$ by $b_1\inv b_2:\operatorname{Spec}L_2\to
\operatorname{Spec}L_1$, and there is a commutative  fibered product diagram 
\begin{equation}\label{E:Ch-Em-Pr}
\xymatrix@R=.8em{
J^{2n+1}_{a,X/K}\ar@{<-}[rr] \ar[dd] \ar@{<-}[rd]&&J^{2n+1}_{a,X_{L_1}/L_1}\ar@{-}[d]  \ar@{<-}[rr]&&J^{2n+1}_a(X_{b_1}) \ar[dd]\\
&J^{2n+1}_{a,X_{L_2/L_2}} \ar[dd] \ar[ru] \ar@{<-}[rr]&\ar[d]&J^{2n+1}_a(X_{b_2}) \ar[dd] \ar[ru]&\\
\operatorname{Spec}K \ar@/^1pc/[uu]^{\delta_Z} &\ar[l]&\operatorname{Spec}L_1   \ar@{-}[l]_<>(0.5){b_1} \ar@/^1pc/[uu]^<>(0.3){\delta_{Z_{L_1}}}  \ar@{<-}[r]&\ar@{-}[r]&\operatorname{Spec}\mathbb C\\
&\operatorname{Spec}L_2 \ar[lu]^{b_2} \ar@/^1pc/[uu]^<>(0.8){\delta_{Z_{L_2}}} \ar[ru] \ar@{<-}[rr]&&\operatorname{Spec}\mathbb C \ar[ru]_{\sigma}&\\
}
\end{equation}
where $J^{2n+1}_{a,X/K}$ and $\delta_Z:\operatorname{Spec}K\to J^{2n+1}_{a,X/K}$ are defined from the rest of the diagram via fibered product.
\end{pro}

\begin{rem}\label{R:DNF-K}
Proposition \ref{P:DNF-K} allows one to define  \emph{the distinguished model} 
$J^{2n+1}_{a,X/K}$ of the image of the Abel--Jacobi map, and  
\emph{the distinguished normal function} $\delta_Z$ associated to 
a cycle class $Z\in\operatorname{CH}^{n+1}(X)$, without first  needing to 
specify a particular inclusion $K\hookrightarrow \mathbb C$.  
\end{rem}

\begin{proof} From the diagram  \eqref{E:Ch-Em-Pr} it is clear that it suffices to establish 
	that there is a commutative fibered product diagram
\begin{equation}\label{E:Ch-Em-Pr-1}
\xymatrix@R=2em{
J^{2n+1}_{a,X_{L_2}/L_2}\ar[r] \ar[d]& J^{2n+1}_{a,X_{L_1}/L_1} \ar[d]\\
\operatorname{Spec}L_2 \ar[r]^{b_1^{-1}b_2} \ar@/^1pc/[u]^<>(0.5){\delta_{Z_{L_2}}} & \operatorname{Spec}L_1 \ar@/_1pc/[u]_<>(0.5){\delta_{Z_{L_1}}} 
}
\end{equation}
\emph{i.e.}, it is enough to focus on the sub-diagram that is the left hand face of the cube in diagram \eqref{E:Ch-Em-Pr}.
 We break the proof into two parts.  First we establish the result for the distinguished models, 
 and second for the distinguished normal functions.

\vskip .2 cm \noindent \emph{Step 1\,: The distinguished models}. 
Let $C$ be a 
 geometrically integral curve over $K$ (admitting a $K$-point)   and let $\gamma \in
  \operatorname{CH}^{n+1}(C\times_{K}
  X)$  be a correspondence  such that, for $i=1,2$,   the
  induced morphisms of complex tori
${\gamma_i}_*: \mathsf J(\mathsf C_i) \rightarrow
\mathsf J^{2n+1}(\mathsf X_{i})$
have respective images equal to $J_a^{2n+1}(\mathsf X_{i})$ (\cite[Prop.~1.1]{ACMVdmij}). Here, 
 $\mathsf X_i$ and $\mathsf C_i$ are the complex
analytic spaces associated to the pull backs of
$X$ and  $C$ to $\operatorname{Spec}\mathbb C$ via the given inclusions $b_i : K \hookrightarrow
\mathbb C$. 
Thus, for $i = 1,2$,
we obtain short exact sequences
$$
\xymatrix@R=.01em{
  0\ar[r]& \mathsf P_i \ar[r]& \mathsf J(\mathsf C_i) \ar@{->}[r]^<>(0.5){}&  
\mathsf J^{2n+1}_a(\mathsf X_i) \ar[r] & 0
}
$$
where $\mathsf P_i$ is defined to be the kernel of the morphism of
complex tori induced by $\gamma$. Moreover, we have seen in Remark
\ref{R:DM-ind-C-gg} that  $\mathsf P_i$  descends to an abelian scheme
$P_i$ over $L_i$.  This gives short exact sequences
$$
\xymatrix@R=.01em{
  0\ar[r]& P_i \ar[r]& J(C_{L_i}) \ar@{->}[r]^<>(0.5){}&  
 J^{2n+1}_{a,X_{L_i}/{L_i}} \ar[r] & 0
 }
$$
defining the distinguished models $J^{2n+1}_{a,X_{L_i/L_i}}$. We want to show 
that the distinguished models differ by base change over $b_1^{-1}b_2:\operatorname{Spec}L_2\to \operatorname{Spec}L_1$. 
We will do this by showing that $P_1$ and $P_2$ differ by base change over $b_1^{-1}b_2$.

   Let $P_{1,L_2}\subseteq J(C_{L_2})$ be the base change of $P_1$ to
   $L_2$.  To show that $P_{1,L_2}=P_2$, it suffices  to show that for all
   natural numbers $N$, the $N$-torsion of $P_{1,L_2}$ and $P_2$ are
   equal\,; \emph{i.e.}, $P_{1,L_2}[N]=P_2[N]$.  But the  $N$-torsion  $P_i[N]$
    is equal to the kernel of the morphism ${\gamma_{i*}}:
   H^1(C_{b_i},\mmu_{N}) \to  H^{2n+1}(X_{b_i},\mmu_{N}^{\otimes (n+1)})$\,; see
   \cite[(2.3)]{ACMVdmij}. 
These are related by the diagram
$$\xymatrix@R=1em{H^1(C_{b_1},\mmu_{N}) \ar[d]_{\simeq}^{\sigma^*}
\ar[r]^{(\gamma_1)_*\qquad } & H^{2n+1}(X_{b_1},\mmu_{N}^{\otimes (n+1)})
\ar[d]_{\simeq}^{\sigma^*} \\
H^1(C_{b_2},\mmu_{N}) \ar[r]^{(\gamma_2)_*\qquad } & H^{2n+1}(X_{b_2},\mmu_{N}^{\otimes
(n+1)})
}$$
implying that $P_{1,L_2}[N]=P_2[N]$, completing the proof.

\vskip .2 cm \noindent \emph{Step 2\,: The distinguished normal functions}.  
We now show that the distinguished normal functions $\delta_{Z_{L_1}}$ and
$\delta_{Z_{L_2}}$ fit into the fibered product diagram
\eqref{E:Ch-Em-Pr-1}\,; \emph{i.e.}, that they agree under base change.   To
begin, recall that the distinguished normal function $\delta_{Z_{L_i}}$ is
characterized by the condition that
$(\delta_{Z_{L_i}})_{an}=\nu_{(Z_{b_i})_{an}}$\,; \emph{i.e.}, the analytic map induced
by $\delta_{Z_L}$ agrees with the analytic normal function (Remark \ref{R:Uni-DNF}).
Algebraically, this is the condition that $(\delta_{Z_{L_i}})_{b_i}=AJ\circ
w_{Z_{b_i}}$ (see \S \ref{S:DNF}).
In other words, to complete the proof of the theorem, it suffices to
show that $(\delta_{Z_{L_1}})_{b_2}=(\delta_{Z_{L_2}})_{b_2}$.
Put differently, by virtue of the fact that $J^{2n+1}_a(X_{b_1})$ and $J^{2n+1}_a(X_{b_2})$ 
have been  identified in Step 1 via base change over $\sigma:\operatorname{Spec}\mathbb C
\to \operatorname{Spec}\mathbb C$, it suffices to show that the outer rectangle of the diagram 
\begin{equation}\label{E:AJ-sig-com}
\xymatrix@R=1.3em{
(\operatorname{Spec}\mathbb C)_\cx(\mathbb C)\ar[r]^{w_{ Z_{b_1}}}
\ar@{->}[d]^\sigma&\operatorname{A}^{n+1}(X_{b_1})\ar@{->}[r]^<>(0.5){AJ}
\ar[d]^<>(0.5){\sigma^*}&J^{2n+1}_a(X_{b_1})(\mathbb C)
\ar[d]^<>(0.5){\sigma(\mathbb C)}\\
(\operatorname{Spec}\mathbb C)_\cx(\mathbb C)\ar[r]^{w_{ Z_{b_2}}}&\operatorname{A}^{n+1}(X_{b_2})\ar@{->}[r]^<>(0.5){AJ}&J^{2n+1}_a(X_{b_2})(\mathbb C)\\
}
\end{equation}
is commutative.  In other words, it suffices to show that $\sigma AJ(Z_{b_1})=AJ((Z_{b_2}))$.

To do this, we  first show the  commutativity of the right hand side of
\eqref{E:AJ-sig-com} on torsion (see also Remark \ref{R:AJ-sig-com}).  
For this one considers the diagram
$$\xymatrix@R=1.3em{
\operatorname{A}^{n+1}(X_{b_1})[N] \ar[r]^{AJ\ }  \ar[d]^{\sigma^*[N]}  
& J^{2n+1}_a(X_{b_1})[N] \ar[d]^{\sigma(\mathbb C)[N]} \ar@{^{(}->}[r]
& H^{2n+1}(X_{b_1}, \mmu_N^{\otimes (n+1)}) \ar[d]^{\sigma^*}\\
\operatorname{A}^{n+1}(X_{b_2})[N] \ar[r]^{AJ \ }
& J^{2n+1}_a(X_{b_2})[N] \ar@{^{(}->}[r]
& H^{2n+1}(X_{b_2}, \mmu_N^{\otimes (n+1)})
}
$$
which is commutative for all integers $N>1$, due to the fact that by construction the
right-hand square is commutative, and the fact that the outer square is commutative
because the composition of horizontal arrows is the Bloch map, which is
functorial with respect to automorphisms of the field (see \emph{e.g.}, \cite[\S 2.3]{ACMVdmij}).

 Now, since $Z$ is defined over $K$, there exist by
\cite[Thm.~1]{ACMVabtriv} an abelian variety $A$ over~$K$, a $K$-point $p\in A(K)$,
and a correspondence $\Xi \in \operatorname{CH}^{n+1}(A\times_K X)$ such that $Z
= \Xi_p - \Xi_0$. Since the Abel--Jacobi map is a regular homomorphism, the base change of $\Xi$ along $b_i$ induces a
homomorphism $\psi_{\Xi,i} : A_{b_i} \to J^{2n+1}_{a}(X_{b_i})$ with $\psi_{\Xi,i} (q) = AJ( \Xi_q - \Xi_0)$, in
particular $(\psi_{\Xi})_i (p) = AJ(Z_{b_i})$. We have then a not \emph{a priori}  commutative diagram
$$\xymatrix@R=1.3em{
	A_{b_1} \ar[r]^<>(0.5){\psi_{\Xi,1}}  \ar[d]^{\sigma^*}  & J^{2n+1}_a(X_{b_1})
	\ar[d]^{\sigma}\\
	A_{b_2} \ar[r]^<>(0.5){\psi_{\Xi,2}} & J^{2n+1}_a(X_{b_2}).
}
$$
However, since the right hand side of \eqref{E:AJ-sig-com} is commutative on torsion, 
this diagram is also commutative on torsion  (note that $\Xi_q - \Xi_0$ is torsion 
in $\operatorname{A}^{n+1}(X_{b_i})$ whenever $q$ is a torsion point in $A_{b_i}$\,; \emph{e.g.} \cite[Lem.~3.2]{ACMVdmij})
and since torsion points are dense, the diagram is in fact commutative, thereby
establishing the desired identity $\sigma AJ(Z_{b_1}) = AJ(Z_{b_2})$.
\end{proof}

According to the  Bloch--Beilinson 
philosophy (see \emph{e.g.} \cite[Lecture~3]{green-spread}), if $X$ is a smooth projective variety defined over a subfield $K\subseteq \cx$, then the kernel of the Abel--Jacobi map  $AJ : \operatorname{CH}^{n+1}(X)_{\mathrm{hom}} \to \operatorname{CH}^{n+1}(X_\cx)_{\mathrm{hom}} \to J^{2n+1}(X_\cx) $ defined on homologically trivial cycles defined over $K$  should be independent of the choice of embedding $K\hookrightarrow \cx$, after tensoring with $\rat$. Even such a concrete consequence of the  Bloch and Beilinson conjectures remains wide open. As a noteworthy consequence of Proposition \ref{P:DNF-K}, we can establish this, with integral coefficients, for algebraically trivial cycle classes\,:

\begin{cor}\label{Cor:fieldemb}
Let $X$ be a smooth projective variety over a field $K$ of finite transcendence degree over $\rat$, and let $Z\in \operatorname{A}^{n+1}(X)$ be an algebraically trivial cycle class. Let $b : K \hookrightarrow \cx$ be an inclusion of fields. Then $AJ(Z_b) = 0$ for one such embedding if and only if $AJ(Z_b) = 0$ for all such embeddings. \qed
\end{cor}

\begin{rem}\label{R:AJ-sig-com}
In fact, one can also show that the right hand side of  \eqref{E:AJ-sig-com} is commutative on all cycle classes (\emph{i.e.}, not just the ones defined over $K$). 
In particular, the kernel of the Abel--Jacobi map restricted to algebraically trivial cycles (not necessarily defined over $K$) is independent of 
the choice of an embedding $K\hookrightarrow \cx$, in the sense that if $Z \in \operatorname{A}^{n+1}(X_\cx)$ is such that $AJ(Z) = 0$, then $AJ(Z^\sigma) = 0$ for all automorphisms $\sigma \in \mathrm{Aut}(\cx)$. For brevity, we have omitted the proof.  
\end{rem}

\subsection{Distinguished models and distinguished normal functions of very
general fibers}\label{S:DM-DNF}

We now focus on the main case of interest in this paper.  
Let $F\subseteq \mathbb C$ be a field of finite  transcendence degree over
$\mathbb Q$.  
Let  $f:X\to B$ be a smooth surjective projective morphism of smooth integral
schemes of finite type over $F$,  and let $n$ be a nonnegative  integer.     Let
$\mathsf f:\mathsf X\to \mathsf B$ be the associated map of complex manifolds.  
Let $\eta$ be the generic point of $B$ with residue field $K$, which is also of
finite transcendence degree over $\mathbb Q$.

\begin{exa}[Distinguished model of a very general fiber]\label{E:DM-CCFam}
In the notation above, 
fix an inclusion $K\subseteq  \mathbb C$, let $X_{\eta}$ be the generic fiber,
and let $J^{2n+1}_{a,X_{\eta}/K}$ be the corresponding distinguished model
associated to $(X_{\mathbb C})_{an}$ (see also Proposition \ref{P:DNF-K}). 
Now let  $\mathsf u\in \mathsf B$ be an $F$-very general point\,; \emph{i.e.},
$\mathsf u$ corresponds to a closed $\mathbb C$-point 
$u:\operatorname{Spec}\mathbb C\to B_{\mathbb C}$, which is itself a morphism
of $\mathbb C$-schemes,  so that the composition  
$u:\operatorname{Spec}\mathbb C\to B_{\mathbb C}\to B$ has image the generic
point of $B$, given by a second  inclusion $i:K\hookrightarrow \mathbb C$.    
From Proposition \ref{P:DNF-K}, 
the distinguished
model of $\mathsf J^{2n+1}(\mathsf X_{\mathsf u})$ is (after pull back to~$K$) isomorphic over $K$ to
$J^{2n+1}_{a,X_{\eta}/K}$.  In fact,   the distinguished models of all $F$-very
general fibers agree up to isomorphism over $K$. Put another way, for any $F$-very 
general point $\mathsf u\in \mathsf B$, corresponding to a point $u:\operatorname{Spec}\mathbb C\to B_{\mathbb C}$, 
\begin{equation*}\label{E:DM-VGF}
 ((J^{2n+1}_{a,X_{\eta}/K})_{ u})_{an}=\mathsf J^{2n+1}_a(\mathsf
X_{\mathsf u}).
\end{equation*}
\end{exa}

\begin{exa}[Distinguished normal function of a very general
fiber]\label{E:NF-CCFam}
In the same situation as Example \ref{E:DM-CCFam}, let $Z\in
\operatorname{CH}^{n+1}(X)$ be such that every Gysin fiber is algebraically
trivial.  Let $\delta_Z:\operatorname{Spec}K\to J^{2n+1}_{a,X_\eta/K}$ be the
distinguished normal function  \eqref{E:Def-DNF}, which \emph{a priori} depends on the inclusion $K\subseteq \mathbb C$.
 However, from Proposition \ref{P:DNF-K}, the distinguished normal function associated to any 
$F$-very general point $\mathsf u\in \mathsf B$ (corresponding to an inclusion 
$i:K\hookrightarrow \mathbb C$)  agrees (after pull back to $K$) with
$\delta_Z$. 
Put another way, for any $F$-very general point $\mathsf u\in \mathsf B$, 
corresponding to a point $u:\operatorname{Spec}\mathbb C\to B_{\mathbb C}$,  
$$
((\delta_Z)_{u})_{an}=(\nu_{\mathsf Z})_{\mathsf
u}:\mathsf
u\to \mathsf J^{2n+1}_a(\mathsf X_{\mathsf u})
$$ 
where $\nu_{\mathsf Z}:\mathsf B\to \mathsf J^{2n+1}(\mathsf X/\mathsf B)$ is
the  analytic normal function associated to $Z$.
\end{exa}

\begin{rem}[The geometric generic fiber]\label{R:Geom-Gen-Fib}
Another way to frame the relationship between the distinguished models and
distinguished normal functions associated to different $F$-very general points
of $\mathsf B$ is to observe that the $\mathbb C$-scheme $X_u$ associated to an 
$F$-very general fiber $\mathsf X_{\mathsf u}$ of $\mathsf f:\mathsf X\to
\mathsf B$ is isomorphic as a $K$-scheme to a geometric generic fiber
$X_{\bar{\cx(B)}}$.  In fact,  after choosing a $K$-isomorphism $\alpha:\mathbb
C\to 
\bar{\cx(B)}$, we have that $X_u$ and $X_{\bar{\cx(B)}}$ are isomorphic over
$\alpha$ (see \emph{e.g.},
\cite[Lem.~2.1]{vial13}). 
\end{rem}

\begin{rem}[First consequence for zero loci of normal functions\,: Conjecture~\ref{CJ:ANFMK} in the algebraically $F$-motivated case] \label{R:FC-NF}
Already, we obtain a quick proof of much of Corollary \ref{C:ZL}, \emph{i.e.}, that that
the zero-locus of $\nu_{{\mathsf Z}}$ is a countable union of algebraic
subsets of $\mathsf B$ defined over $F$. 
Indeed, from Example~\ref{E:NF-CCFam} we have that if $(\nu_{\mathsf Z})_{\mathsf u}$
is zero for one $F$-very general point $\mathsf u\in \mathsf B$, 
then it is zero for every $F$-very general point $\mathsf u'\in \mathsf B$ (see also Corollary \ref{Cor:fieldemb}).
By continuity, if $\nu_{\mathsf Z}$ is not identically zero, then  the
zero locus of $\nu_{\mathsf Z}$ is contained in the complement of the $F$-very
general points, which is a countable union of algebraic subsets of $\mathsf B$ defined
over
$F$, and not equal to $\mathsf B$. 
Restricting $\nu_{\mathsf Z}$ to an $F$-desingularization of each
irreducible component and arguing recursively on each component, one obtains the claim.  
Note that together with Conjecture \ref{CJ:GG}, one obtains a
proof of Corollary~\ref{C:ZL}. We will, however, give below a direct proof of
Theorem~\ref{T:ANF-MK}, and hence of Corollary~\ref{C:ZL}, that does not rely on
the validity of Conjecture~\ref{CJ:GG}.
\end{rem}

\section{Spreading the distinguished model and distinguished normal function}

We now consider a family of smooth complex projective varieties, and descend the
image of the Abel--Jacobi map of a very general fiber to the generic point of
the base of the family.  We then spread this to an open subset of the base.  The
following theorem collects some properties of this spread.   (Note
that the caveat of Remark \ref{R:Jnotation} already applies to the
abelian scheme $\mathsf J^{2n+1}_a(\mathsf X_{\mathsf U}/\mathsf U)$
constructed below.)

\begin{teo}\label{T:ANF-MK-LG}
Let $F\subseteq \mathbb C$ be a subfield of finite  transcendence degree over
$\mathbb Q$, 
let $f:X\to B$ be a smooth surjective projective morphism of smooth integral
varieties of finite type over $F$,  and let $n$ be a nonnegative integer.   Let
$\mathsf f:\mathsf X\to \mathsf B$ be the associated morphism of complex
analytic spaces.

Let $\eta$ be the generic point of $B$ with residue field $K$, and fix an
inclusion $K\hookrightarrow  \mathbb C$.  
Let $X_{\eta}$ be the
generic fiber of $X$ over $K$,  
let $J^{2n+1}_{a,X_{\eta}/K}$ be the
distinguished model of $\mathsf J^{2n+1}_a((X_{\mathbb C})_{an})$ over $K$  and
let  
\begin{equation}\label{E:Gam-N-Def}
\Gamma \in
\operatorname{CH}^{\dim
(J^{2n+1}_{a,X_{\eta}/K})+n}(J^{2n+1}_{a,X/K}\times_{K} 
X_{\eta}) \ \ \ \text{ and } \ \ \ M\in \mathbb Z_{>0}
\end{equation}
be the correspondence and integer, respectively, from Theorem \ref{T:DMIJ-0} (see also Proposition \ref{P:DNF-K}).
Spread this data to a Zariski open subset $U\subseteq B$. More precisely, let
$U\subseteq B$ be a Zariski open subset over which 
there is an abelian scheme $$g:J^{2n+1}_{a,X_U/U}\to U$$ with generic fiber
isomorphic to $J^{2n+1}_{a,X_{\eta}/\eta}$, and a cycle $$\Gamma_U \in
\operatorname{CH}^{\dim
(J^{2n+1}_{a,X_{\eta}/K})+n}(J^{2n+1}_{a,X_U/U}\times_{U}
X_{U})$$ with $(\Gamma_U)_{\eta}=\Gamma$.
Let $\mathsf U\subseteq \mathsf B$ be the Zariski open subset corresponding to
$U\subseteq B$, and let $\mathsf \Gamma_{\mathsf U}$ be the corresponding complex analytic 
correspondence.

\begin{enumerate}
\item  For every prime number $\ell$  the correspondence $\Gamma_U$ induces a
morphism of  sheaves on $U$
\begin{equation}\label{E:Lisse}
(\Gamma_U)_*:R^1 g_*\mathbb Q_\ell \stackrel{}{\longrightarrow} 
R^{2n+1}(f|_U)_* \mathbb Q_\ell(n)
\end{equation}
which at the  geometric generic  point $\bar u:\operatorname{Spec}\bar K\to
\operatorname{Spec}K\to U$,  
induces an inclusion of $\operatorname{Gal}(\bar K/K)$-representations
\begin{equation}\label{E:DMIJ-Repr}
(\Gamma_{U,\bar u})_* : H^1((J^{2n+1}_{a,X_{\eta}/K})_{\bar K},\rat_\ell)
\hookrightarrow H^{2n+1}(X_{\bar
K},\rat_\ell(n))
\end{equation}
(with image $\coniveau^nH^{2n+1}(X_{\bar K},\rat_\ell(n))$, where $\coniveau^\bullet$ denotes the geometric coniveau filtration).

\item  The correspondence  $\mathsf \Gamma_{\mathsf U}$ induces a morphism of
variations of pure integral Hodge structures and thus a morphism of
relative complex tori over $\mathsf U$
\begin{equation}\label{E:DM-Spd}
(\mathsf \Gamma_{\mathsf U})_*:(J^{2n+1}_{a,X_{U}/U})_{an}\to \mathsf
J^{2n+1}(\mathsf X/\mathsf B)|_{\mathsf U}.
\end{equation}
The image of \eqref{E:DM-Spd} is an algebraic relative complex torus $\mathsf
A_{\mathsf U}\subseteq \mathsf J^{2n+1}(\mathsf X/\mathsf B)|_{\mathsf U}$ over
$\mathsf U$, induced by an abelian scheme $$A_U/U,$$ defined over $F$, with
generic fiber $(A_U)_{\eta}$ isomorphic to $J^{2n+1}_{a,X_{\eta}/K}$ over $K$.

For $F$-very general $\mathsf u\in \mathsf U$, the morphism \eqref{E:DM-Spd} 
restricts to a morphism 
\begin{equation}\label{E:DM-Repr-CC}
((\mathsf \Gamma_{\mathsf U})_*)_{\mathsf u}:\mathsf J^{2n+1}_{a}(\mathsf X_{\mathsf
u})\to \mathsf J^{2n+1}(\mathsf X_{\mathsf u})
\end{equation}
that  is given by $M \cdot \mathsf i^{2n+1}_{a,\mathsf X_{\mathsf u}}$, \emph{i.e.}, $M$ times the natural inclusion (where $M$ is defined in  \eqref{E:Gam-N-Def}).  In particular, the image
of \eqref{E:DM-Repr-CC}, \emph{i.e.}, the fiber $\mathsf A_{\mathsf U,\mathsf u}$,  is
equal to $\mathsf J^{2n+1}_a(\mathsf X_{\mathsf u})$.

\item

Let $Z\in \operatorname{CH}^{n+1}(X)$ be a cycle class with every Gysin fiber
algebraically trivial, let $Z_{\eta}$ be the restriction of $Z$ to the generic fiber $X_\eta$, and let 
\begin{equation*}\label{E:NF-etaB}
\delta_{Z_{\eta}}:\operatorname{Spec}K\to J^{2n+1}_{a,X_{\eta}/K}
\end{equation*}
be the associated distinguished normal function  (see \eqref{E:Def-DNF} and Proposition \ref{P:DNF-K}). 
After possibly replacing the Zariski open subset $U\subseteq B$   with a smaller
Zariski open subset, let $\delta:U\to J^{2n+1}_{a,X_U/U}$ be the spread of the
distinguished normal function, 
and let $\delta_{an} :\mathsf U \to (J^{2n+1}_{a,X_U/U})_{an}$  denote the
associated morphism of complex analytic spaces.    We have the  following
formula relating the normal function  $\nu_{\mathsf Z}$, the spread $\delta_{an}$ of
the
distinguished normal function,  and the morphism \eqref{E:DM-Spd}\,:
\begin{equation}\label{E:NFanNFalg}
(\mathsf \Gamma_{\mathsf U})_*\circ \delta_{an}  = M\cdot  \nu_{\mathsf Z}|_{\mathsf
U},
\end{equation}
and $ M\cdot \nu_{\mathsf Z }|_{\mathsf U}$ is algebraic, and defined over $F$. 
\end{enumerate}

\end{teo}

\begin{proof}
(1) Correspondences induce morphisms of sheaves, giving \eqref{E:Lisse}. 
\eqref{E:DMIJ-Repr} is just a statement about fibers of correspondences, and
follows from Theorem \ref{T:DMIJ-0} (but see also \cite[Thm.~A]{ACMVdmij}).

(2) Correspondences induce morphisms of Hodge structures, giving
\eqref{E:DM-Spd}.  
To show  that the image $\mathsf A_{\mathsf U}$ of  \eqref{E:DM-Spd} is
algebraic, one shows that the kernel of the morphism of relative complex tori is
algebraic.  For this it suffices to check that torsion is preserved by
$\operatorname{Aut}(\mathbb C/F)$, and it is easy to see this holds  from the
fact that the morphism is induced by an algebraic cycle defined over $F$ 
(as in Lemma \ref{L:Gamm-M})\,; alternatively, it is dominated by the
relative algebraic complex torus $(J^{2n+1}_{a,X_U/U})_{an}$.

The assertion \eqref{E:DM-Repr-CC} is just a statement about fibers of
correspondences, and Theorem \ref{T:DMIJ-0} and Example \ref{E:DM-CCFam} provide
the needed identification of fibers. One also uses the general observation that
the image of the multiplication by $M$ map is the same complex torus.   

The  final statement, that the generic fiber $(A_U)_{\eta}$ is isomorphic to
$J^{2n+1}_{a,X_{\eta}/K}$ over $K$, can be established as follows.  Let $G$ be
the kernel of the $K$-isogeny  
$J^{2n+1}_{a,X_{\eta}/K}\to  (A_U)_{\eta}$. The isogeny, when pulled back to
$\mathsf u$, gives a morphism
$$
\xymatrix{
\mathsf J^{2n+1}_a(\mathsf X_{\mathsf u}) \ar@{->>}[r]& \mathsf A_{\mathsf
U,\mathsf u}  \ar[r]^<>(0.5)\cong & \mathsf J^{2n+1}_a(\mathsf X_{\mathsf u})
}
$$
with composition equal to  the multiplication by $M$ map. Thus $G$ and
$J^{2n+1}_{a,X_{\eta}/K}[M]$ are reduced $K$-subschemes of
$J^{2n+1}_{a,X_{\eta}/K}$ with the same $\mathbb C$-points, and are therefore
the same scheme.  It follows that $(A_U)_{\eta}\cong
J^{2n+1}_{a,X_{\eta}/K}/J^{2n+1}_{a,X_{\eta}/K}[M]=J^{2n+1}_{a,X_{\eta}/K}$.

(3) 
The only thing to show is \eqref{E:NFanNFalg}.   Since both sides of the
equation are continuous functions, it suffices to prove the assertion for a
dense subset of $\mathsf U$, and in particular we can focus on  $F$-very
general points $\mathsf u\in \mathsf U$.  
The assertion then follows from  (2)  together with Example \ref{E:NF-CCFam}.
\end{proof}

We now use Theorem \ref{T:ANF-MK-LG} to prove Theorem \ref{T:ANF-MK} over a
Zariski dense open subset of the base\,:

\begin{cor}\label{C:ANF-MK}
Let 
$\mathsf f:\mathsf X\to \mathsf B$ be a smooth 
surjective
projective
morphism of complex algebraic manifolds,  let $n$ be a nonnegative
integer, let $\mathsf J^{2n+1}(\mathsf X/\mathsf B)\to \mathsf B$ be
the $(2n+1)$-st relative Griffiths intermediate Jacobian.  There is a Zariski
open subset $\mathsf U\subseteq \mathsf B$,   a relative algebraic complex
subtorus $\mathsf J^{2n+1}_a(\mathsf X_{\mathsf U}/\mathsf U)\subseteq \mathsf
J^{2n+1}(\mathsf X/\mathsf B)|_U$ over $\mathsf U$ such that for very general
$\mathsf u\in \mathsf U$ the fiber $\mathsf J_a(\mathsf X_{\mathsf U}/\mathsf
U)_{\mathsf u}\subseteq \mathsf J^{2n+1}(\mathsf X_{\mathsf u})$ is the image
$\mathsf J^{2n+1}_a(\mathsf X_{\mathsf u})$ of the Abel--Jacobi map 
$\mathsf {AJ}_{\mathsf X_{\mathsf u}}:\operatorname{A}^{n+1}(\mathsf X_{\mathsf
u})\to \mathsf J^{2n+1}(\mathsf X_{\mathsf u})$, 
and for any 
$\mathsf Z\in \operatorname{CH}^{n+1}(\mathsf X)$ with every Gysin fiber
algebraically trivial\,:
\begin{enumerate}
\item The restriction of the normal function  $
\nu_{\mathsf Z}|_{\mathsf U}:\mathsf U\to \mathsf J^{2n+1}(\mathsf X/\mathsf
B)|_{\mathsf U}$ has image contained in the relative algebraic complex torus 
$\mathsf J^{2n+1}_a(\mathsf X_{\mathsf U}/\mathsf U)$ and is an algebraic map.  
\item If, moreover, $\mathsf X$, $\mathsf B$, $\mathsf f$, and $\mathsf Z$ are
all defined over a subfield $F\subseteq \mathbb C$, then so are $\mathsf
J^{2n+1}_a(\mathsf X_{\mathsf U}/\mathsf U)$ and the morphisms $\mathsf
J^{2n+1}_a(\mathsf X_{\mathsf U}/\mathsf U)\to \mathsf U$ and $\nu_{\mathsf
Z}|_{\mathsf U}$.
\end{enumerate}

\end{cor}

\begin{proof}
In case (1), since  $\mathsf f:\mathsf X\to \mathsf B$ is defined over some
field $F\subseteq \mathbb C$ that is finitely generated over $\mathbb Q$, we may
as well make this assumption from the start.  In case (2) we may take our field
of definition $F'$ to be contained in the given field $ F$, and can   base
change to $F$ at the end, if necessary, and so we may as well assume $F$ is
finitely generated over $\mathbb Q$ in case (2), as well.  We are then in the
situation of 
Theorem \ref{T:ANF-MK-LG},  and we will use  the notation from that theorem
moving forward. 

First, we can take  $\mathsf J^{2n+1}_a(\mathsf X_{\mathsf U}/\mathsf
U)=\mathsf A_{\mathsf U}\subseteq \mathsf J^{2n+1}(\mathsf X/\mathsf
B)|_{\mathsf U}$, the image of  \eqref{E:DM-Spd}. Let $A_U$ be the corresponding
abelian scheme over $U\subseteq B$.  
From Theorem \ref{T:ANF-MK-LG}(2) we have that the generic fiber  $(A_U)_{\eta}$
is isomorphic over $K$ to the distinguished model $J^{2n+1}_{a,X_{\eta}/K}$. 
With cycle class $Z\in \operatorname{CH}^{n+1}(X)$ as in the theorem, let 
\begin{equation}\label{E:C-ANF-Eq}
\delta_{Z_{\eta}}:\operatorname{Spec}K\to J^{2n+1}_{a,X_{\eta}/K}\cong
(A_U)_{\eta}
\end{equation}
be the distinguished normal function.  This then spreads to an $F$-morphism
$\delta:U\to A_U$, 
after possibly replacing $U$ with a smaller Zariski open subset.
The associated complex analytic map $\delta_{an}:\mathsf U\to \mathsf A_{\mathsf
U}\subseteq \mathsf J^{2n+1}(\mathsf X/\mathsf B)|_{\mathsf U}$ has the property
that for $F$-very general $\mathsf u\in \mathsf U$ we have $(\delta_{an})_{\mathsf
u}:\mathsf u\to \mathsf A_{\mathsf U,\mathsf u}=\mathsf J^{2n+1}_a(\mathsf
X_{\mathsf u})$ has image $\nu_{\mathsf Z}(\mathsf u)$\,; \emph{i.e.}, it agrees with the
complex analytic normal function.  This follows from \eqref{E:C-ANF-Eq} and
Example \ref{E:NF-CCFam}.  Therefore, since $\delta_{an}$ and $\nu_{\mathsf
Z}|_{\mathsf
U}$ are continuous and agree on the dense open subset of $F$-very general points
of $\mathsf U$, they agree on all of  $\mathsf U$. 
\end{proof}

\section{Extending the distinguished model and distinguished normal function}\label{S:Ext-DM-DNF}

In light of Corollary \ref{C:ANF-MK}, in order to prove Theorem \ref{T:ANF-MK},
we only need to show that the distinguished model and distinguished normal
function extend over the entire base $\mathsf B$.  We do this now.

\begin{proof}[Proof of Theorem \ref{T:ANF-MK}]
We use the notation from Theorem \ref{T:ANF-MK}, and the partial result,
Corollary~\ref{C:ANF-MK}.  As mentioned above, we only need to show that the
spread of the  distinguished model $\mathsf J^{2n+1}_a(\mathsf X_{\mathsf
U}/\mathsf U)\subseteq \mathsf J^{2n+1}(\mathsf X/\mathsf B)|_{\mathsf U}$  and
distinguished normal function $\delta_{an}:\mathsf U\to \mathsf J^{2n+1}_a(\mathsf
X_{\mathsf U}/\mathsf U)$ extend over the entire base $\mathsf B$, to give
algebraic objects over $F$.  To begin, we switch to the algebraic setting, and
let $J^{2n+1}_{a,X_U/U}/U$ be the algebraic model of $\mathsf J^{2n+1}_a(\mathsf
X_{\mathsf U}/\mathsf U)$, and let $\delta:U\to
J^{2n+1}_{a,X_U/U}$
 be the
associated morphism of $F$-schemes.  

First, we show that $J^{2n+1}_{a,X_U/U}$ extends to an abelian
scheme $g: \tilde J^{2n+1}_{a,X/B}\to B$ over $B$.  If $\dim B = 1$, we use
the
inclusion \eqref{E:DMIJ-Repr} and the N\'eron--Ogg--Shafarevich
criterion as in \cite[Lem.~6.1(a)]{ACMV14}.
If $\dim B\ge 2$, using
the dimension $1$ case we can extend over the generic points of
divisors in the boundary $B- U$, and thus we can assume
that $\operatorname{codim}_B(B-U)\ge 2$.  The
assertion now follows from the Faltings--Chai Extension Theorem
\cite[Cor.~6.8, p.185]{faltingschai}.

Next we show that the relative algebraic complex torus $\tilde{\mathsf
J}^{2n+1}_a(\mathsf X/\mathsf B):=  (\tilde J^{2n+1}_{a,X /B})_{an}$ induces  an
algebraic relative subtorus ${\mathsf J}^{2n+1}_a(\mathsf X/\mathsf B)\subseteq
\mathsf J^{2n+1}(\mathsf X/\mathsf B)$ extending ${\mathsf J}^{2n+1}_a(\mathsf
X_{\mathsf U}/\mathsf U)$.  For this we use the basic fact that any morphism of 
variations of Hodge structures extends over a locus of codimension at least $2$.  (Indeed, by purity, the natural map $\pi_1(\mathsf U, \mathsf u) \to \pi_1(\mathsf B, \mathsf u)$ is an isomorphism\,; now use \cite[Thm.~10.11, p.243]{PS08}, or the proof of \cite[Lem.~6.3, p.117]{hain95}.)
It now follows that the  inclusion $\mathsf J^{2n+1}_a(\mathsf X_{\mathsf
U}/\mathsf U)\subseteq \mathsf J^{2n+1}(\mathsf X/\mathsf B)|_{\mathsf U}$
extends to a morphism
\begin{equation}\label{E:JaX/BDef}
\tilde {\mathsf J}^{2n+1}_a(\mathsf X/\mathsf B)\longrightarrow \mathsf
J^{2n+1}(\mathsf X/\mathsf B)
\end{equation}
with finite kernel, which \emph{a priori} may be nontrivial only over $\mathsf
B-\mathsf U$.  We define ${\mathsf J}^{2n+1}_a(\mathsf X/\mathsf B)\subseteq
\mathsf J^{2n+1}(\mathsf X/\mathsf B)$ to be the image of \eqref{E:JaX/BDef}.  
Although it is not needed, we note that by Zariski's Main Theorem  the morphism
of relative algebraic complex tori  $\tilde {\mathsf J}^{2n+1}_a(\mathsf X/\mathsf B)\to 
{\mathsf J}^{2n+1}_a(\mathsf X/\mathsf B)$ is an isomorphism.

The normal function $\nu_{\mathsf Z}$ has image contained in $\mathsf
J^{2n+1}_a(\mathsf X/\mathsf B)$,  since $\mathsf J^{2n+1}(\mathsf X/\mathsf B)$
is separated, and $\nu_{\mathsf Z}|_{\mathsf U}$ has image contained in $\mathsf
J^{2n+1}_a(\mathsf X/\mathsf B)$, which is closed in $\mathsf J^{2n+1}(\mathsf
X/\mathsf B)$.
Finally, it is straight forward to check that $\nu_{\mathsf Z}$ is algebraic,
and defined over $F$, since $\nu_{\mathsf Z}|_{\mathsf U}$ is algebraic and
defined over~$F$.  For completeness, we include this last assertion as Lemma
\ref{L:UtoXext-1} below.
\end{proof}

\begin{lem}\label{L:UtoXext-1}
Let $X,Y$ be schemes of finite type over $F\subseteq \mathbb C$, with $X$ 
reduced and $Y$ separated,  let $U\subseteq X$ be a Zariski open subset, 
let $f_U:U\to Y$ be a morphism of $ F$-schemes, and assume that the associated
morphism of analytic spaces $(f_U)_{an}:\mathsf U\to \mathsf Y$ extends to a
morphism $\mathsf f: \mathsf X\to \mathsf Y$.  Then $f_U$ extends to  a morphism
$f:X\to Y$ over $ F$ with $f_{an}=\mathsf f$. 
\end{lem}

\begin{proof}
We may immediately reduce to the case with $X$ integral, and $Y$ reduced.   Now
consider the graph $\Gamma_{f_U}\subseteq U\times_ F Y\subseteq X\times_ FY$,
which is closed in $U\times _ FY$ since $Y$ is separated.   Let $\Gamma\subseteq
X\times_ FY$ be the closure of $\Gamma_{f_U}$, which we observe is  an  integral
subscheme.    Now let $\mathsf \Gamma=\Gamma_{an}\subseteq \mathsf X\times
\mathsf Y$ be the associated complex analytic space.  We have by assumption that
$(\Gamma_{f_U})_{an}=(\mathsf \Gamma_{\mathsf f})|_{\mathsf U\times \mathsf Y}$.
Now $\mathsf \Gamma_{\mathsf f}$ is the analytic closure of $(\mathsf
\Gamma_{\mathsf f})|_{\mathsf U\times \mathsf Y}$ in $\mathsf X\times \mathsf
Y$.  Since $(\mathsf \Gamma_{\mathsf f})|_{\mathsf U\times \mathsf Y} 
=(\Gamma_{f_U})_{an}\subseteq \mathsf \Gamma$, we have $\mathsf \Gamma_{\mathsf
f}$ is equal to the analytic closure of $(\Gamma_{f_U})_{an}$ in $\mathsf
\Gamma$.   But $\Gamma_{f_U}$ is a Zariski open subset of an integral scheme
$\Gamma$ of finite type over $ F$, and so $\mathsf \Gamma_{\mathsf f}=\mathsf
\Gamma$. (In general, if $T$ is a locally closed subset of a scheme $Z/\cx$ which is locally of finite type, then $T$ is dense in $Z$ if and only if $\mathsf T$ is dense in $\mathsf Z$ \cite[Expos\'e XII, Cor.~2.3]{sga1}.)

Now we just need to conclude that $\Gamma$ induces a morphism $X\to Y$. 
It suffices to show that the second projection $q_1:\Gamma \to X$ is an
isomorphism. 
But this follows from the fact that a morphism between the complex analytic
spaces associated to two $ F$-schemes descends to $ F$ if and only if it is
$\mathrm{Aut}(\cx/ F)$-equivariant (apply, \emph{e.g.}, \cite[\S 5.2]{VoisinHodgeLoci} to its graph), and the fact that $\mathsf q_1:\mathsf
\Gamma\to \mathsf X$ is
an isomorphism.
\end{proof}

\begin{rem}
  \label{R:Jnotation}
The notation $\mathsf J^{2n+1}_a(\mathsf X/\mathsf B)$
may be slightly misleading, in the sense that formation of this object
is not compatible with base change in $\mathsf B$.  While the very general
fiber $\mathsf J^{2n+1}_a(\mathsf X/\mathsf B)_{\mathsf u}$ is equal to the
image of the Abel--Jacobi map $\mathsf J^{2n+1}_a(\mathsf X_{\mathsf u})$, in
some cases there is a countably infinite union of algebraic subsets of $\mathsf B$  
 over which
the geometric coniveau of the fiber  $\coniveau ^nH^{2n+1}(\mathsf X_{\mathsf
b},\mathbb Q)$ jumps.
 If this is the case, then over these points the fiber
$\mathsf J^{2n+1}_a(\mathsf X/\mathsf B)_{\mathsf b}$ is strictly contained
in $\mathsf J^{2n+1}_a(\mathsf X_{\mathsf b})$.   Nonetheless, we feel
that $\mathsf J^{2n+1}_a(\mathsf X/\mathsf B)$ is good notation in the sense
that this is the smallest relative algebraic subtorus of $\mathsf
J^{2n+1}(\mathsf X/\mathsf B)$ that interpolates between the very general
$\mathsf J^{2n+1}_a(\mathsf X_{\mathsf u})$.  Moreover, by part (1) of 
Theorem \ref{T:ANF-MK}, it serves as a target for every algebraically motivated normal function.
\end{rem}

%%%%%%%%%%%%%%%%%%%%%%%%%%%%%%%%%%%%%%%%%%% --- BIB ---
%%%%%%%%%%%%%%%%%%%%%%%%%%%%%%%%%%

\bibliographystyle{amsalphadoi}
\def\cprime{$'$}
\providecommand{\bysame}{\leavevmode\hbox to3em{\hrulefill}\thinspace}
\providecommand{\MR}{\relax\ifhmode\unskip\space\fi MR }
\providecommand{\MRhref}[2]{%
	\href{http://www.ams.org/mathscinet-getitem?mr=#1}{#2}
}
\providecommand{\href}[2]{#2}

\end{document}